\documentclass[11pt]{amsart}
\usepackage{pb-diagram,amssymb,epic,eepic,verbatim,graphicx}
\usepackage{graphics,epsfig,psfrag}

\newtheorem{theorem}{Theorem}[section]
\newtheorem{corollary}[theorem]{Corollary}

\newtheorem{proposition}[theorem]{Proposition}

\theoremstyle{definition}
\newtheorem{definition}[theorem]{Definition}
\newtheorem{convention}[theorem]{Convention}
\theoremstyle{remark}
\newtheorem{remark}[theorem]{Remark}

\newcommand\cB{\mathcal{B}}
\newcommand\cC{\mathcal{C}}

\newcommand\cE{\mathcal{E}}

\newcommand\cH{\mathcal{H}}
\newcommand\cI{\mathcal{I}}
\newcommand\cJ{\mathcal{J}}

\newcommand\cM{\mathcal{M}}

\newcommand\cP{\mathcal{P}}

\newcommand\cS{\mathcal{S}}

\newcommand{\cU}{\mathcal{U}}
\newcommand{\cV}{\mathcal{V}}
\newcommand{\cW}{\mathcal{W}}

\newcommand{\N}{\mathbb{N}}
\newcommand{\R}{\mathbb{R}}

\newcommand{\Z}{\mathbb{Z}}
\newcommand{\Q}{\mathbb{Q}}

\newcommand\eps{\epsilon}

\newcommand{\on}{\operatorname}
\newcommand{\Ham}{\on{Ham}}
\newcommand{\graph}{\on{graph}}
\newcommand{\Ind}{ \on{Ind}}

\newcommand\reg{{\on{reg}}}

\def\dbar{\overline\partial}
\newcommand\ul{\underline}
\newcommand\mtimes{\times \kern-.3ex}

\begin{document}

\title{Quilted Floer trajectories with constant components}

\author{Katrin Wehrheim and Chris T. Woodward}

\address{Department of Mathematics,
Massachusetts Institute of Technology,
Cambridge, MA 02139.
{\em E-mail address: katrin@math.mit.edu}}

\address{Department of Mathematics, 
Rutgers University,
Piscataway, NJ 08854.
{\em E-mail address: ctw@math.rutgers.edu}}

\begin{abstract}  
We fill a gap in the proof of the transversality result for quilted
Floer trajectories in \cite{quiltfloer} by addressing trajectories for
which some but not all components are constant.
Namely we show that for generic sets of split Hamiltonian
perturbations and split almost complex structures, the moduli spaces of
parametrized quilted Floer trajectories of a given index are smooth of expected
dimension.
An additional benefit of the generic split Hamiltonian perturbations is that they perturb the given cyclic Lagrangian correspondence such that any geometric composition of its factors is transverse and hence immersed.
\end{abstract} 

\maketitle

\vspace{-5mm}

\section{Introduction}  

Quilted Floer homology is defined for a cyclic generalized Lagrangian
correspondence $\ul{L}$, that is, a sequence of symplectic manifolds
$M_0,M_1,\ldots, M_{r},M_{r+1}$ with $M_0=M_{r+1}$ for $r \ge 0$, and
a sequence of Lagrangian correspondences
$$ 
L_{01} \subset M_0^- \times M_1, \ \ L_{12} \subset M_1^- \times
M_2, \ \ \ldots, \ \ L_{r(r+1)} \subset M_{r}^- \times M_{r+1} .
$$ 
For the purpose of transversality arguments we do not need any
monotonicity assumptions as in \cite{quiltfloer}, so throughout we
merely assume that all Lagrangian correspondences are compact.
Quilted Floer homology $HF(\ul{L})$ can be defined as the
standard Floer homology of a pair of Lagrangians in the product
manifold
$M_0^-\times M_1\times M_2^- \times \ldots \times M_{r}$, given by
products of the $L_{i(i+1)}$. (For even $r$ one adds a diagonal to the sequence before making this construction.)
As such, $HF(\ul{L})$ depends on the choice of a Hamiltonian function and almost complex structure on this
product manifold, which generically would not be of split form - i.e.\ induced by a tuple of Hamiltonian functions and almost complex structures on each symplectic manifold $M_j$. The quilted definition
of $HF(\ul{L})$ in \cite{quiltfloer} on the one hand generalizes this construction by allowing a choice of widths $\ul{\delta}=(\delta_j>0)_{j=0,\ldots,r}$ of the strips mapping to each $M_j$. 
On the other hand, we claim in \cite{quiltfloer} that the quilted Floer complex can be constructed (in particular transversality can be achieved) for Hamiltonians and almost complex structures of split type. 
That is, we restrict our choice of perturbation data to a tuple of Hamiltonian functions and a tuple of almost complex structures in the complete metric spaces
$$ 
\cH_t(\ul{\delta}) := \oplus_{j=0}^r \cC^\infty([0,\delta_j] \times M_j, \R) , \qquad
\cJ_t(\ul{\delta}) := \oplus_{j=0}^r \cC^\infty([0,\delta_j],\cJ(M_j,\omega_j)) ,
$$ 
where $\cJ(M_j,\omega_j)$ is the space of smooth $\omega_j$-compatible
almost complex structures on $M_j$.  While this split form is not
necessary for the strip shrinking analysis in \cite{isom}, it is
particularly helpful for constructing relative invariants (such as the
functor associated to a correspondence in \cite{cat}) from more
complicated quilted surfaces, which cannot be interpreted as single
surface mapping to a product manifold.  Unfortunately, the
transversality proof in \cite{quiltfloer} for the quilted Floer
trajectory spaces for generic split perturbation data
$\ul{H}\in\cH_t(\ul{\delta})$ and $\ul{J}\in\cJ_t(\ul{\delta})$ has a
significant gap:
It fails to explicitly discuss trajectories $\ul{u}=( u_j : \R \times [0,\delta_j] \to M_j )_{j=0,\ldots,r}$ for which some but not all components are constant.  This intermediate situation is not an easy combination of the two extreme cases 
(all components nonconstant, or all components constant)
as we seem to claim in \cite{quiltfloer}.

\subsection*{Results}

In Theorem~\ref{quilt trans} below we complete the proof of 
the transversality claimed in \cite{quiltfloer} by working with a more specific set of generic split Hamiltonian perturbations which may be of independent interest. 
In Theorem~\ref{H trans} and Corollary~\ref{cor H trans} we find a
dense open subset of $\cH_t(\ul{\delta})$ for any given cyclic
Lagrangian correspondence such that, after perturbation by one of
those split Hamiltonian diffeomorphisms, any geometric composition of
its factors is transverse and hence immersed.  Starting from such a
Hamiltonian perturbation, we observe that quilted Floer trajectories
with constant components induce quilted Floer trajectories for a
shorter cyclic generalized Lagrangian correspondence, given by a
localized version of geometric composition across the constant strips.
Using this point of view, we are able to find generic sets of split
almost complex structures for which quilted Floer trajectories with
constant components are regular, as well.
In fact, we show that quilted Floer trajectories with constant components
are very rare as summarized in Remark~\ref{rmk trans} and sketched below.

\subsection*{Idea of Proof}
A key role in the proof is played by certain families of isotropic subspaces which arise in the proof of transversality for the universal moduli space of almost complex structures and Floer trajectories.
The elements of the cokernel of the linearized operator of the universal moduli space are tuples of $-J_i$-holomorphic sections $\eta_i$ of $u_i^*TM_i$ for $i=0,\ldots,r$ with Lagrangian seam conditions determined by the tangent spaces of $L_{i(i+1)}$.  Ignoring Hamiltonian perturbations for simplicity, the problem of constant components occurs for example when some $u_j$ is constant (with value say $x_j \in M_j$) but the adjacent components $u_{j-1},u_{j+1}$ are non-constant. 
Then variations in the almost complex structures prove vanishing of $\eta_{j-1},\eta_{j+1}$, and hence $\eta_j:\R\times[0,\delta_j]\to T_{x_j} M_j$ is $-J_j$-holomorphic with boundary conditions in
\begin{align*}
\Lambda_j(s) 
&:= 
{\Pr}_{T_{x_j}M_j} \bigl( T_{(u_{j-1}(\delta_{j-1},s), x_j)} L_{(j-1)j} \cap ( \{ 0 \} \times T_{x_j} M_j) \bigr) 
\; \subset T_{x_j} M_j , \\
\Lambda_j'(s) 
&:= 
{\Pr}_{T_{x_j}M_j} \bigl(T_{(x_j, u_{j+1}(0,s))} L_{j(j+1)} \cap ( T_{x_{j}} M_j \times \{ 0 \})\bigr) 
\,\;\quad \subset
T_{x_j} M_j .
\end{align*}
The spaces $\Lambda_j(s),\Lambda_j'(s)$ are isotropic spaces varying with $s\in\R$,
despite the fact that $u_j\equiv x_j$ is constant.
We can now proceed differently in three non-exclusive cases.

\begin{enumerate}
\item
The easiest case is to assume that $\Lambda_j(s),\Lambda_j'(s)$ are $s$-independent. We may then enlarge these isotropic spaces to constant Lagrangian subspaces and deduce that $\eta_j$ lies in the kernel of an operator $\partial_s + A$, where $A$ is an $s$-independent self-adjoint operator and invertible (since by choice of $\ul{H}$ the generators of the Floer complex are cut out transversally). We then deduce vanishing of $\eta_j$ from the general fact that operators of this form $\partial_s+A$ are isomorphisms.

\item
An intermediate case occurs when $\Lambda_j(s)$ or $\Lambda_j'(s)$ fails to be Lagrangian (i.e.\ have maximal dimension) for some $s\in\R$.  For example, if $L_{(j-1)j}$ is the graph of a symplectomorphisms, then the intersection $\Lambda_j$ is trivial.  We show that this case of a quilted Floer trajectory with constant component does not occur for generic $(J_i)_{i\neq j}$.

\item
The most difficult case occurs when $\Lambda_j(s)$ and $\Lambda_j'(s)$ are non-constant families of Lagrangian subspaces.  We show that for generic $\ul{H}$ the locus on which such varying Lagrangian subspaces are possible is of positive codimension in the space of boundary values $(u_{j-1}(\delta_{j-1},s),u_{j+1}(0,s))$. 
Then we again exclude this case for generic $(J_i)_{i\neq j}$.
\end{enumerate}

Thus, for generic Hamiltonian perturbations $\ul{H}$ and almost complex structures $\ul{J}$ we in fact show a splitting property for any quilted Floer trajectory with constant components, namely along the seam $(u_j(s,\delta_j),u_{j+1}(s,0))\in L_{j(j+1)}$ we have
$T L_{j(j+1)} = \Lambda_j \times \Lambda_{j+1}$, where $\Lambda_j\subset TM_j$ is a constant Lagrangian subspace given as above, and $\Lambda_{j+1}$ is the $s$-dependent projection of $T L_{j(j+1)} \cap (\{0\}\times T M_{j+1} )$.
For a precise statement see Remark~\ref{rmk trans}.

The arguments in case (b) and (c) crucially rely on the following interpretation of quilted Floer trajectories with constant components as quilted Floer trajectories for a generalized Lagrangian correspondence obtained by a local version of geometric composition. 
If $\ul{u}=(u_0,\ldots, u_r)$ is a solution with $u_j\equiv x_j$ as above, then
$(u_0,\ldots,u_{j-1},u_{j+1},\ldots,u_r)$ is a quilted Floer trajectory for the generalized correspondence $(L_{01}, \ldots, L_{(j-1)j} \circ L_{j(j+1)},\ldots, L_{r(r+1)}) $.
We show in Theorem~\ref{H trans} that, after a generic Hamiltonian perturbation of $\ul{L}$, any geometric composition $L_{(j-1)j} \circ L_{j(j+1)}$ is an immersed Lagrangian correspondence. It becomes embedded if we restrict to values in $M_j$ near $x_j$. 
Hence $(u_0,\ldots,u_{j-1},u_{j+1},\ldots,u_r)$ can be viewed as quilted Floer trajectory for a smooth generalized Lagrangian correspondence.

We showed in \cite{isom} that transversality for this composed correspondence implies transversality for the original correspondence for sufficiently small widths $\delta_j>0$. 
Here we extend this transversality to solutions with constant $u_j$ for arbitrary $\delta_j>0$ and generic perturbation data $\ul{H}, \ul{J}$.

\subsection*{Alternative approaches}

It is perhaps worth remarking that all of the correspondences intended as applications in \cite{quiltfloer,fielda,tri} fit into the easiest case (a) described above since these Lagrangian correspondences $L_{01} \subset M_0^- \times M_1$ are {\em quasisplit} in the following sense:
The intersection $( T_{x_0} M_0 \times \{ 0 \}) \cap T_{(x_0,x_1)} L_{01} $ is independent of $x_1$ and the intersection $(\{ 0 \} \times T_{x_1} M_1) \cap T_{(x_0,x_1)}L_{01}$
is independent of $x_0$. 
Examples are split correspondences $L_0 \times L_1$, graphs of symplectomorphisms, correspondences arising from fibered coisotropics, and the embedded geometric composition of any two quasisplit correspondences.  If all Lagrangian correspondences are quasisplit then the simple argument in case (a) above completes the transversality argument for the universal moduli space in \cite{quiltfloer}.

Note however that one can easily construct Lagrangian correspondences
that are not quasisplit by applying a nonsplit Hamiltonian
diffeomorphism of $M_0^-\times M_{1}$ to a split correspondence
$L_0\times L_1$. 

Another possibility for achieving transversality at quilted Floer trajectories with constant components is to introduce nonsplit perturbations as in \cite{per:gys} and \cite{ll:geom}.
However, in order to implement this perturbation scheme for more general relative quilt invariants, one would have to replace each seam with seam condition $L_{ij}\subset M_i\times M_j$ by a novel triple seam condition pairing the two patches in $M_i$ and $M_j$ via a diagonal with one boundary of a strip in $M_i\times M_j$, whose other boundary has boundary conditions in $L_{ij}$. In that setup we may use non-split perturbations on the strip.

\medskip

{\em We thank Max Lipyanski for pointing out the question of constant components.}

\section{Hamiltonian perturbations of generalized Lagrangian correspondences}
\label{Ham}

Given a cyclic generalized Lagrangian correspondence $\ul{L}=(L_{j(j+1)})_{j=0,\ldots,r}$, widths $\ul{\delta}=(\delta_j>0)_{j=0,\ldots,r}$, and a tuple of Hamiltonian functions $\ul{H}=( H_j )_{j = 0,\ldots,r}\in\cH_t(\ul{\delta})$, the generators of the quilted Floer complex are tuples of Hamiltonian chords,
\begin{equation*} \label{cI}
\cI(\ul{L} ,\ul{H}) := \left\{ \ul{x}=\bigl(x_j: [0,\delta_j] \to M_j\bigr)_{j=0,\ldots,r} \, \left|
\begin{aligned}
\dot x_j(t) = X_{H_j}(x_j(t)), \\
(x_{j}(\delta_j),x_{j+1}(0)) \in L_{j(j+1)} 
\end{aligned} \right.\right\} .
\end{equation*}
They are canonically identified, via   
$\ul{x}\mapsto (x_{0}(\delta_0),x_{1}(0),x_{1}(\delta_1),\ldots, x_{r}(\delta_r),x_{0}(0) )$, with the fiber product
\begin{align*}
& \times_{\phi_{\delta_0}^{H_0}}\bigl( L_{01} \times_{\phi_{\delta_1}^{H_1}} L_{12} \ldots \times_{\phi_{\delta_r}^{H_r}} L_{r(r+1)}\bigr) \\
&:=
\bigl(L_{01}\times L_{12}\times\ldots\times L_{r(r+1)}\bigr)
\cap 
\bigl( \graph(\phi_{\delta_1}^{H_1}) \times \graph(\phi_{\delta_2}^{H_2}) \times \ldots \times \graph(\phi_{\delta_0}^{H_0}) \bigr)^T ,
\end{align*}
where $\phi_{\delta_j}^{H_{j}}$ is the time $\delta_j$ Hamiltonian
flow of~$H_{j}$ and $(\ldots)^T$ denotes the exchange of factors
$M_1\times\ldots\times M_0\times M_0 \to M_0\times M_1
\times\ldots\times M_0$.  In this setting we proved in
\cite{quiltfloer} that Hamiltonians of split type suffice to achieve
transversality for the generators.  We now strengthen this result to
achieve transversality for all partial fiber products.

\begin{convention}
Here and in the following we use indices $j\in\N$ modulo $r+1$.
A pair of indices $j<j'$ denotes a pair $j,j'\in\N$ with $j < j' \leq j+r+1$.
A pair of indices $j\lhd j'$ denotes a pair $j,j'\in\N$ with $j+1 < j' \leq j+r+1$,
that is, with at least one other index between $j$ and $j'$. 

For any proper subset $I\subset\{0,\ldots,r\}$ let $I^C\subset\{0,\ldots,r\}$ be its complement. 
Then a consecutive pair of indices $j<j' \in I^C$ (resp.\ $j\lhd j'\in I^C$)
denotes a pair  $j<j'$ (resp.\ $j\lhd j'$) as above such that 
$j,j'\in I^C$ and $\{j+1,\ldots,j'-1\}\subset I$.
\end{convention}

\begin{definition} 
For any pair of indices $j\lhd j'$ we define the partial fiber product
\begin{align*}
& L_{j(j+1)} \times_{H_{j+1}} L_{(j+1)(j+2)} \ldots \times_{H_{j'-1}} L_{(j'-1)j'} \\
&:=
\bigl( L_{j(j+1)}\times \ldots\times L_{(j'-1)j'} \bigr)
\cap 
\bigl( M_j \times \graph(\phi_{\delta_{j+1}}^{H_{j+1}}) \times\ldots\times \graph(\phi_{\delta_{j'-1}}^{H_{j'-1}}) \times M_{j'} \bigr).
\end{align*}
We trivially extend this notation to the case $j'=j+1$ by $L_{j j'}=L_{j(j+1)}$.
For a general proper subset of indices $I\subset\{0,\ldots,r\}$ we then define the partial fiber product 
$$
\mtimes_{I,\ul{H}}\ul{L} \;  := \prod_{\text{consec.}\; j<j'\in I^C}
L_{j(j+1)} \times_{H_{j+1}} \ldots \times_{H_{j'-1}} L_{(j'-1)j'}
$$
to be the product of the above fiber products
for each consecutive pair of indices $j< j' \in I^C$.
We view the intersection $\cI(\ul{L} ,\ul{H}) =\mtimes_{\{0,\ldots,r\},\ul{H}}\ul{L}$ 
as the full fiber product case $I=\{0,\ldots,r\}$.
\end{definition}

\begin{theorem}  \label{H trans}
There is a dense open subset 
$\Ham^*(\ul{L})\subset\cH_t(\ul{\delta})$ so that for every $\ul{H}\in\Ham^*(\ul{L})$ the defining equations for $\mtimes_{I,\ul{H}}\ul{L}$ for any $I\subset\{0,\ldots,r\}$ are transversal.
\end{theorem}  

\begin{proof}
Each of the fiber products under consideration is of the following form: It is the set of tuples 
$(m'_0, m_1, m'_1, \ldots , m'_r, m_{r+1})\in L_{01}\times  \ldots\times L_{r(r+1)}$ satisfying 
\begin{equation}\label{kdefine}
 \phi_{\delta_i}^{H_i}(m_i) = m_i' \quad\forall i\in I .
\end{equation} 
It suffices to find a dense open subset of regular Hamiltonians for
each of these problems, since the intersection of finitely many dense
open subsets remains dense and open. So we fix some choice of
$I\subset\{0,\ldots,r\}$ and consider the universal moduli $\cM^{\on{univ}}$ space
of data $(H_0,\ldots,H_r,m_0',m_1,\ldots,m_r',m_{r+1})$ satisfying
\eqref{kdefine}, where now each $H_j$ has class $\cC^\ell$ for some
$\ell > \sum_{i\in I^C}\dim M_i$.  It is cut out by the diagonal
values of the $\cC^\ell$-map
\begin{align*}
 L_{01}\times L_{12}\ldots \times L_{r(r+1)} \times 
 \bigoplus_{k=0}^r \cC^\ell([0,\delta_k] \times M_k)
&\longrightarrow \bigoplus_{j\in I} M_j\times M_j, \\ 
(m_0',\ldots,m_{r+1},H_0,\ldots,H_r)& \longmapsto ( \phi_{\delta_i}^{H_i}(m_i),m_i')_{i\in I} .
\end{align*}
The linearized equations for $\cM^{\on{univ}}$ are
\begin{equation} \label{surj} 
v_i' - D \phi_{\delta_i}^{H_i} (h_i,v_i) = 0 \in TM_i
\qquad \forall i \in I .
\end{equation}
for $v_i \in T_{m_i} M_i$, $v'_i \in T_{m'_i} M_i$,  and $h_i \in \cC^\ell([0,\delta_i] \times M_i)$.
The map
$$ \cC^\ell([0,\delta_i] \times M_i) \to T_{\phi_{\delta_i}^{H_i}(m_i)}
M_i, \quad h_i \mapsto D \phi_{\delta_i}^{H_i} (h_i,0) $$
is surjective, which shows that the product of the operators on the
left-hand side of \eqref{surj} is also surjective.
So by the implicit function theorem $\cM^{\on{univ}}$ is a $\cC^\ell$ Banach manifold, 
and we consider its projection to $\oplus_{k=0}^r \cC^\ell([0,\delta_k] \times M_k)$.  
This is a Fredholm map of class $\cC^\ell$ and index $\sum_{i\in I^C}\dim M_i$ (in particular $0$ for the full intersection $I=\{0,\ldots,r\}$).
Hence, by the Sard-Smale theorem, the set of regular values (which coincides with the set of functions $\ul{H}=(H_0,\ldots,H_r)$ such that the perturbed intersection is
transversal) is dense in $\oplus_{k=0}^r \cC^\ell([0,\delta_k] \times M_k)$.
Moreover, the set of regular values is open for each $\ell > \sum_{i\in I^C}\dim M_i$.
Indeed, by the compactness of $L_{01}\times L_{12}\ldots \times L_{r(r+1)}$, a $\cC^1$-small change in $\ul{H}$ leads to a small change in perturbed intersection points, with small change in the linearized operators.

Now, by approximation of $\cC^\infty$-functions with $\cC^\ell$-functions, the set of regular values in $\oplus_{k=0}^r \cC^\infty([0,\delta_k] \times M_k)$ is dense in the $\cC^\ell$-topology for all $\ell > \sum_{i\in I^C}\dim M_i$, and hence dense in the $\cC^\infty$-topology.
Finally, the set of regular smooth $\ul{H}$ is open in the $\cC^\infty$-topology as a special case of the $\cC^1$-openness.
\end{proof} 

We now reformulate this Theorem by using the Hamiltonian flows of $\ul{H}$ to perturb the Lagrangian correspondences and then applying a geometric composition in some factors.

\begin{corollary} \label{cor H trans}
For $\ul{H}\in{\rm Ham}^*(\ul{L})$ the perturbed generalized correspondence
$$
\ul{L}' := \Bigl( 
L':= \bigl({\rm Id_{M_j}}\times \phi_{\delta_{j+1}}^{H_{j+1}}\bigr)L_{j(j+1)} \Bigr)_{j=0,\ldots,r} 
$$ 
has the following intersection and composition properties:
\begin{enumerate}
\item
The generalized intersection
$$
\cI(\ul{L}',0)=
(L'_{01}\times\ldots\times L'_{r(r+1)})\cap ( \Delta_{M_1} \times \ldots \times \Delta_{M_0})^T
$$ 
is transverse and canonically identified with $\cI(\ul{L},\ul{H})$.
\item
For any proper subset $I\subset\{0,\ldots,r\}$ the partial fiber product $\mtimes_{I,0}\ul{L}'$
is cut out transversally (and canonically identified with $\mtimes_{I,\ul{H}}\ul{L}$). It is a product of the transverse intersections
\begin{align*}
&L'_{j(j+1)} \times_{\Delta_{j+1}} L'_{(j+1)(j+2)} \ldots \times_{\Delta_{j'-1}} L'_{(j'-1)j'} \\
&:= 
\bigl( L'_{j(j+1)} \times L'_{(j+1)(j+2)} \ldots \times L'_{(j'-1)j'} \bigr) \cap \bigl( M_j\times\Delta_{M_{j+1}} \ldots \Delta_{M_{j'-1}}\times M_{j'} \bigr) \\
& =
\bigl({\rm Id}_{M_j}\times \phi_{\delta_{j+1}}^{H_{j+1}}\times{\rm Id}_{M_{j+1}} \ldots\times \phi_{\delta_{j'-1}}^{H_{j'-1}} \times \phi_{\delta_{j'}}^{H_{j'}}\bigr)
\bigl( L_{j(j+1)} \times_{H_{j+1}} \ldots \times_{H_{j'-1}} L_{(j'-1)j'} \bigr) 
\end{align*}
for consecutive pairs of indices $j<j'\in I^C$.
\item
By a direct generalization of \cite[Lemma 2.0.5]{quiltfloer}, the projection 
$$
\Pi_{M_j\times M_{j'}} : \; L'_{j(j+1)} \times_{\Delta_{j+1}} \ldots \times_{\Delta_{j'-1}} L'_{(j'-1)j'}
\; \longrightarrow \; M_j\times M_{j'}
$$
is an immersion onto the geometric composition 
$L'_{j(j+1)} \circ \ldots \circ L'_{(j'-1)j'} \subset M_j\times M_{j'}$.
\end{enumerate}
\end{corollary}

We will in particular be interested in this composition near a fixed point in $M_{j+1}\times\ldots\times M_{j'-1}$ given by the components of an intersection point in $\cI(\ul{L}',0)$. For any such point there is a neighbourhood $\cU\subset M_{j+1}\times\ldots\times M_{j'-1}$ such that the projection $\Pi_{M_j\times M_{j'}}$ embeds $\bigl( L'_{j(j+1)} \times_{\Delta_{j+1}} L'_{(j+1)(j+2)} \ldots \times_{\Delta_{j'-1}} L'_{(j'-1)j'} \bigr) \cap \bigl(  M_j \times \cU \times M_{j'} \bigr)$ into $M_j \times M_{j'}$.
This is a localized version of the embedded geometric composition
(as studied in \cite{quiltfloer}) of the perturbed Lagrangians.
We will be using the following analogue perturbed geometric composition of the unperturbed Lagrangian correspondences.

\begin{definition} 
For a proper subset $I\subset\{0,\ldots,r\}$ and $\ul{x}\in\cI(\ul{L},\ul{H})$ we define the locally composed cyclic Lagrangian correspondence $\ul{L}^{I,\ul{H},\ul{x}}$ between the underlying manifolds $(M_j)_{j\in I^C}$ to be the cyclic sequence consisting of 
$L_{jj'}^{\ul{H},\ul{x}}\subset M_j\times M_{j'}$ for each consecutive pair of indices $j< j' \in I^C$, given by 
\begin{align*}
L_{jj'}^{\ul{H},\ul{x}} 
&:=
\Pi_{M_j\times M_{j'}}\bigl(\bigl(L_{j(j+1)} \times_{H_{j+1}} L_{(j+1)(j+2)} \ldots \times_{H_{j'-1}} L_{(j'-1)j'}\bigr) \cap \tilde\cU_{\ul{x},j,j'} \bigr) 
\end{align*}
for $\tilde\cU_{\ul{x},j,j'}:=M_j \times \cU_{\ul{x},j,j'} \times M_{j'}$, 
where $ \cU_{\ul{x},j,j'}$ is a neighbourhood of $\big(x_{j+1}(0),x_{j+1}(\delta_{j+1}), \ldots,\linebreak[2] x_{j'-1}(0),x_{j'-1}(\delta_{j'-1})\bigr)$ such that $\Pi_{M_j\times M_{j'}}$ is injective on the intersection.
\end{definition}

\begin{remark}
Given a regular tuple of Hamiltonian functions $\ul{H}\in{\rm Ham}^*(\ul{L})$ as in Theorem~\ref{H trans} and a proper subset $I\subset\{0,\ldots,r\}$, let
$\ul{\delta}^I:=(\delta_j)_{j\in I^C}$ and $\ul{H}^I:=(H_j)_{j\in I^C}$.
Then the transversality assertions of Theorem~\ref{H trans} moreover imply that for any $\ul{x}\in\cI(\ul{L},\ul{H})$ the intersection
$\cI(\ul{L}^{I,\ul{H},\ul{x}}, \ul{H}^I)$ is transverse.
It contains $(x_j)_{j\in I^C}$, and no other points if the neighbourhoods $\cU_{\ul{x},j,j'}$ are chosen sufficiently small.
\end{remark}

In preparation for the analysis of quilted Floer trajectories with
constant components, we next study the lift from
$L_{jj'}^{\ul{H},\ul{x}}$ to $M_{j+1}\times M_{j'-1}$ and its
connection with the intersections 
$TL_{j(j+1)} \cap (\{0\}\times T M_{j+1})$ and $TL_{(j'-1)j'}\cap (T
M_{j'-1}\times\{0\})$. 
A priori, the latter are isomorphic to collections of isotropic subspaces of $TM_{j+1}$ resp.\ $TM_{j'-1}$ parametrized by $L_{j(j+1)}$ resp.\ $L_{(j'-1)j'}$.
As mentioned in the introduction, a first step is to understand
the locus where these subspaces are Lagrangian, and how they may vary along $L_{jj'}^{\ul{H},\ul{x}}$. For that purpose we introduce the following notation.

\begin{definition} \label{def P}
Let $j\lhd j'$ be a pair of indices.
\begin{enumerate}
\item 
We denote by $\cS_{jj'} \subset L_{j(j+1)}\times L_{(j'-1)j'}$ the set of points $\ul{q}=(q_j,q_{j+1},q_{j'-1},q_{j'})$ for which 
$$
\Lambda_{(j+1)(j'-1)}(\ul{q}):= T_{\ul{q}} (L_{j(j+1)}\times L_{(j'-1)j'}) \cap T_{\ul{q}}( \{q_j\}\times M_{j+1}\times M_{j'-1}\times\{q_{j'}\}\bigr)
$$
induces a Lagrangian subspace in $T_{q_{j+1}} M_{j+1}\times T_{q_{j'-1}} M_{j'-1}$ (with the appropriate signs on the symplectic forms).
\item 
Given moreover $\ul{H}\in{\rm Ham}^*(\ul{L})$, $\ul{x}\in\cI(\ul{L},\ul{H})$, we denote by 
$$ 
\cP_{jj'} : \; L_{jj'}^{\ul{H},\ul{x}}  \; \longrightarrow \; M_{j+1}\times M_{j'-1}
$$ 
the composition of the lift from $L_{jj'}^{\ul{H},\ul{x}}$ to $\bigl(L_{j(j+1)} \times_{H_{j+1}} \ldots \times_{H_{j'-1}} L_{(j'-1)j'}\bigr) \cap \tilde\cU_{\ul{x},j,j'}$ and the projection to the second and penultimate component, i.e.\ to a neighbourhood of $(x_{j+1}(0),x_{j'-1}(\delta_{j'-1}))$.
\end{enumerate}
\end{definition}

The following Proposition shows that the set $\cS_{jj'}$ can equivalently be defined as the locus where the linearized Lagrangian correspondences split, and that this splitting locus is closely related to the vanishing of $D\cP_{jj'}$.

\begin{proposition} \label{prop L}
The following holds for any pair of indices $j\lhd j'$.
\begin{enumerate}
\item
$\Lambda_{(j+1)(j'-1)}(\ul{q})$ is Lagrangian if and only if we have splittings
$$
T_{(q_j, q_{j+1})} L_{j(j+1)} = \Lambda_j \times \Lambda_{j+1} , \qquad
T_{(q_{j'-1},q_{j'})} L_{(j'-1)j'} = \Lambda_{j'-1} \times \Lambda_{j'}
$$
into Lagrangian subspaces
$$
\Lambda_j = {\Pr}_{TM_{j}} \bigl( T_{(q_j, q_{j+1})} L_{j(j+1)} \cap (T_{q_j} M_{j}\times \{0\})\bigr),
$$
$$
\Lambda_{j+1} = {\Pr}_{TM_{j+1}} \bigl( T_{(q_j, q_{j+1})} L_{j(j+1)} \cap (\{0\}\times T_{q_{j+1}} M_{j+1}) \bigr),
$$ 
$$
\Lambda_{j'-1}= {\Pr}_{TM_{j'-1}} \bigl( T_{(q_{j'-1}, q_{j'})} L_{(j'-1)j'} \cap 
(T_{q_{j'-1}} M_{j'-1}\times \{0\}) \bigr) ,
$$
$$
\Lambda_{j'}= {\Pr}_{TM_{j'}} \bigl( T_{(q_{j'-1}, q_{j'})} L_{(j'-1)j'} \cap (\{0\}\times T_{q_{j'}} M_{j'}) \bigr) .
$$
\item
The subset $\cS_{jj'} \subset L_{j(j+1)}\times L_{(j'-1)j'}$ is compact.
\item 
For any $\ul{H}\in{\rm Ham}^*(\ul{L})$, $\ul{x}\in\cI(\ul{L},\ul{H})$ the linearization 
$D_{(q_j,q_{j'})}\cP_{jj'}: T_{(q_j,q_{j'})} L_{jj'}^{\ul{H},\ul{x}} \to T_{\cP_{jj'}(q_j,q_{j'})} (M_{j+1}\times M_{j'-1})$ is trivial iff $(q_j,\cP_{jj'}(q_j,q_{j'}),q_{j'})\in  \cS_{jj'}$.
\end{enumerate}
\end{proposition}

\begin{proof}
By definition, $\Lambda_{(j+1)(j'-1)}(\ul{q})$ induces the (automatically isotropic) subspace $\Lambda_{j+1}\times\Lambda_{j'-1} \subset TM_{j+1}\times TM_{j'-1}$. It is Lagrangian iff
both factors have maximal dimension, i.e.\ are Lagrangian.
Moreover, since $TL_{j(j+1)}$ is Lagrangian, maximal dimension of its subspace $\Lambda_{j+1}$ directly implies maximal dimension of $\Lambda_{j}$, and vice versa, and analogously with 
$TL_{(j'-1)j'}$. This proves (a).

The splitting conditions in (a) can be phrased as intersections having maximal dimension, hence are preserved in a limit, so occur on a closed subset of $L_{j(j+1)}\times L_{(j'-1)j'}$. This shows that $\cS_{jj'}$ is closed, so (b) follows directly from the compactness of all Lagrangian submanifolds involved.

Given a point $(q_j,q_{j+1},q'_{j'-1},q_{j'})=(q_j,\cP_{jj'}(q_j,q_{j'}),q_{j'})$ there exist $q'_{j+1},q_{j+2},\ldots,q'_{j'-2},q_{j'-1}$ such that
$(q_j,q_{j+1},q'_{j+1},\ldots,q_{j'-1},q'_{j'-1},q_{j'})\in \bigl(L_{j(j+1)} \times_{H_{j+1}} \ldots \times_{H_{j'-1}} L_{(j'-1)j'}\bigr) \cap \tilde\cU_{\ul{x},j,j'}$. Now $D_{(q_j,q_{j'})}\cP_{jj'}$ is the composition of the lift from $T_{(q_j,q_{j'})} L_{jj'}^{\ul{H},\ul{x}}$ to 
\begin{align}
& T_{(q_j,q_{j+1})}L_{j(j+1)} \times \ldots\times T_{(q'_{j'-1},q_{j'})} L_{(j'-1)j'} \nonumber\\
& \cap \quad  T_{q_{j}}M_j \times\graph( d\phi_{\delta_{j+1}}^{H_{j+1}}(q_{j+1}) ) \times\ldots\times\graph( d\phi_{\delta_{j'-1}}^{H_{j'-1}}(q_{j'-1})  ) \times T_{q_{j'}}M_{j'} 
\label{giant}
\end{align}
and the projection to $T_{q_{j+1}}M_{j+1}\times T_{q'_{j'-1}} M_{j'-1}$. 
Hence $D_{(q_j,q_{j'})}\cP_{jj'}\equiv 0$ if and only if \eqref{giant} is a subset of 
$T_{q_{j}}M_j\times\{0\}\times T_{q'_{j+1}}M_{j+1}\times\ldots\times T_{q_{j'-1}}M_{j'-1} \times\{0\} \times T_{q_{j'}}M_{j'}$.
On the other hand, our choice of $\ul{H}$ guarantees that the projection of \eqref{giant} to $T_{q_{j}}M_{j}\times T_{q_{j'}} M_{j'}$ is injective with Lagrangian image $T_{(q_j,q_{j'})} L_{jj'}^{\ul{H},\ul{x}}$.
So if $D_{(q_j,q_{j'})}\cP_{jj'}\equiv 0$, then both intersections $T_{(q_j,q_{j+1})}L_{j(j+1)} \cap \{0\}\times T_{q_{j+1}} M_{j+1}$
and $T_{(q'_{j'-1},q_{j'})}L_{(j'-1)j'} \cap T_{q'_{j'-1}} M_{j'-1}\times\{0\}$ have maximal dimension. As above that implies $(q_j,\cP_{jj'}(q_j,q_{j'}),q_{j'})\in \cS_{jj'}$.
Conversely, if the latter is true, i.e.\ both
$T_{(q_j,q_{j+1})}L_{j(j+1)}=\Lambda_j\times\Lambda_{j+1}$ and $T_{(q'_{j'-1},q_{j'})} L_{(j'-1)j'}=\Lambda_{j'-1}\times\Lambda_{j'}$ are of split form, then the intersection with the graphs in \eqref{giant} cannot allow for any nonzero vectors in $\Lambda_{j+1}$ or  $\Lambda_{j'-1}$, and hence $D_{(q_j,q_{j'})}\cP_{jj'}\equiv 0$.
\end{proof}

In the following section we will ``generically`` exclude quilted Floer
trajectories with constant components of the following two types:
Firstly, those along whose seam values we have $D\cP_{jj'}\not\equiv 0$
somewhere; secondly, those along whose seam values $D\cP_{jj'}\equiv 0$ but
$\Lambda_{(j+1)(j'-1)}$ varies. This will only leave quilted Floer
trajectories with constant components, for which transversality
follows from transversality for the moduli space of the locally
composed cyclic Lagrangian correspondence.  These arguments require
the following understanding of the structure of the split form set
$\cS_{jj'}$, the variation of the intersection $\Lambda_{(j+1)(j'-1)}$, and the intersection of $\cS_{jj'}$ with lifts of the local compositions $L_{jj'}^{\ul{H},\ul{x}}$.

\begin{theorem}  \label{intersection trans}
The following intersection properties hold for any pair of indices
$j\lhd j'$.
\begin{enumerate}
\item
For any $\ul{q}\in \cS_{jj'}$ there exists an open neighbourhood $\cV_{\ul{q}} \subset L_{j(j+1)}\times L_{(j'-1)j'}$ and smooth functions $G_n: \cV_{\ul{q}} \to \R$ for $n=1,\ldots, N:=\frac{(\dim M_{j}+ \dim M_{j'})(\dim M_{j+1} + \dim M_{j'-1})}{4}$ such that
$$
\cS_{jj'}\cap\cV_{\ul{q}} = \bigcap_{n=1}^N G_n^{-1}(0)  .
$$
Moreover, if $\gamma:(-\eps,\eps)\to \cS_{jj'} \cap \bigl(M_{j}\times \{q_{j+1}\} \times \{q_{j'-1}\}\times  M_{j'}\bigr)$ is a smooth path with $dG_n(\gamma)\equiv 0$ for all $n=1,\ldots N$, then $\Lambda_{(j+1)(j'-1)} (\gamma(t))$ is constant in $t\in(-\eps,\eps)$ as subspace of $T_{q_{j+1}}M_{j+1}\times T_{q_{j'-1}}M_{j'-1}$.
\item
Fix a finite open cover $\cS_{jj'}\subset\bigcup_{\ul{q}\in S_{jj'}} \cV_{\ul{q}}$ by subsets as in (b) with $S_{jj'}\subset\cS_{jj'}$ finite.
Then there is a dense open subset $\cH_{jj'}(\ul{L})\subset\Ham^*(\ul{L})$ such that the following holds: For every $\ul{H}\in\cH_{jj'}(\ul{L})$, $\ul{x}\in\cI(\ul{L},\ul{H})$, $\ul{q}\in S_{jj'}$, and $1\leq n\leq N$ the function $G_{\ul{q},n}^{\ul{H},\ul{x}}: \cV^{\ul{H},\ul{x}}_{\ul{q},n} \to \R$, $(z_j,z_{j'})\mapsto G_n (z_j,\cP_{jj'}(z_j,z_{j'}),z_{j'})$ defined on the open set
$$
\cV^{\ul{H},\ul{x}}_{\ul{q},n}:=\bigl\{ (z_j,z_{j'})\in L_{jj'}^{\ul{H},\ul{x}} \,\big|\, (z_j,\cP_{jj'}(z_j,z_{j'}),z_{j'})\in \cV_{\ul{q}}, \; d G_n (z_j,\cP_{jj'}(z_j,z_{j'}),z_{j'}) \neq 0 \bigr\}
$$
is transverse to $0$.
\end{enumerate}
\end{theorem}

\begin{proof}
Given $\ul{q}=(q_j,q_{j+1},q_{j'-1},q_{j'})\in \cS_{jj'}$ we have splittings 
\begin{align*}
T_{(q_j,q_{j+1})}L_{j(j+1)} &=\Lambda^0_{j}\times \Lambda^0_{j+1} \;\;\, \subset \; T_{q_j} M_j \times T_{q_{j+1}} M_{j+1} , \\
T_{(q'_{j'-1},q_{j'})}L_{(j'-1)j'} &=\Lambda^0_{j'-1}\times \Lambda^0_{j'} \; \subset \; T_{q_{j'-1}} M_{j'-1} \times T_{q_{j'}} M_{j'} 
\end{align*}
into products of Lagrangian subspaces. With this we have
$\Lambda_{(j+1)(j'-1)}(\ul{q}) = \{0\} \times \Lambda^0_{j+1} \times \Lambda^0_{j'-1} \times \{0\}$
and permutation of factors provides an isomorphism
$$
T_{\ul{q}} (L_{j(j+1)}\times L_{(j'-1)j'}) \cong K_0 \times \Lambda_0
\quad\text{with}\;
K_0:=\Lambda^0_{j} \times \Lambda^0_{j'}, \;
\Lambda_0:=\Lambda^0_{j+1} \times \Lambda^0_{j'-1} .
$$ 
Now there exists an open neighbourhood $\cW_{\ul{q}} \subset M_j \times M_j+1\times M_{j'-1}\times M_j'$ of $\ul{q}$ that is symplectomorphic to an open subset of
$K^0 \times \Lambda^0 \times \bigl( K_0 \times \Lambda_0 \bigr)^*$ such that
$\cV_{\ul{q}}:=  \cW_{\ul{q}}\cap (L_{j(j+1)}\times L_{(j'-1)j'})$ corresponds to $\on{graph} d F$ for some smooth function $F:K_0 \times \Lambda_0 \to \R$, restricted to an open subset. We will identify $K_0\cong K_0^*\cong\R^k$, $k=\frac 12(\dim M_{j}+ \dim M_{j'})$ and $\Lambda_0\cong \Lambda_0^*\cong\R^\ell$, $\ell=\frac 12(\dim M_{j+1} + \dim M_{j'-1})$ and use coordinates $(\ul{x},\ul{y}) \in K_0\times\Lambda_0$.
Then the intersection $\Lambda_{(j+1)(j'-1)}$ at some point $(\ul{x},\ul{y},\nabla_{K_0}F (\ul{x},\ul{y}), \nabla_{\Lambda_0}F (\ul{x},\ul{y}))$ is spanned by those vectors
$\bigl(  0 , \ul{b} , 0 , \bigl( \sum_{i=1}^\ell b_i \tfrac{\partial^2 F}{\partial y_i \partial y_l }\bigr)_{l=1,\ldots,\ell} \bigr)$ for which
$\sum_{i=1}^\ell b_i \tfrac{\partial^2 F}{\partial y_i \partial x_\kappa}= 0$ for $\kappa=1,\ldots,k$. 
It is Lagrangian iff its rank is the maximal $\ell$.
Hence $\cS_{jj'} \cap \cV_{\ul{q}}$ is the zero set of 
$$
(G_n)_{n=1,\ldots,N} := \biggl( \frac{\partial^2 F}{\partial y_i \partial x_\kappa}
\biggr)_{i=1,\ldots,\ell, \kappa=1,\ldots,k}  : \; \cV_{\ul{q}} \; \longrightarrow \;\R^{\ell k} = \R^N .
$$
For the second part of (a) we consider a path $\gamma(t)=(\ul{x}(t), 0, \nabla_{K_0}F (\ul{x}(t),0), 0)$ given by a smooth path $\ul{x}:(-\eps,\eps)\to K_0$ such that $\frac{\partial F}{\partial y_i}(\ul{x}, 0) \equiv 0$, $\frac{\partial^2 F}{\partial y_i \partial x_\kappa} (\ul{x}, 0)\equiv 0$, and in particular $\partial_{y_l} G_{n\simeq(i,\kappa)}(\ul{x},0)=\frac{\partial^3 F}{\partial y_l \partial y_i \partial x_\kappa} (\ul{x}, 0)\equiv 0$ for all $l, i, \kappa$.
The latter guarantees that $\tfrac{\partial^2 F}{\partial y_i \partial y_l }(\ul{x}(t),0)$ and hence $\Lambda_{(j+1)(j'-1)} (\gamma(t))$ is independent of $t$.

Approaching (b), note that we may reformulate the claim as transversal intersection of $L_{j(j+1)}\times_{H_{j+1}} \ldots \times_{H_{j'-1}} L_{(j'-1)j'}$ with the zero set of 
$\tilde G_n (\ul{z}):=G_n(z_j,z_{j+1},z_{j'-1}',z_{j'})$ 
on the open set
$$
\left\{ \ul{z}=(z_j,z_{j+1},z'_{j+1},\ldots,z_{j'-1}',z_{j'}) \in L_{j(j+1)}\times \ldots \times L_{(j'-1)j'} \left|
\begin{array}{l} 
(z_j,z_{j+1},z'_{j'-1},z_{j'})\in \cV_{\ul{q}}, \\ 
d G_n (z_j,z_{j+1},z'_{j'-1},z_{j'}) \neq 0 
\end{array}
\right\}\right. .
$$
The universal moduli space of regularity $m\in\N$ for this problem is the preimage of $\{0\}\times\Delta_{M_{j+1}}\times\ldots\Delta_{M_{j'-1}}$ of the map
$$
L_{j(j+1)}\times L_{(j+1)(j+2)}\ldots \times L_{(j'-1)j'} \times 
 \bigoplus_{i=j+1}^{j'-1} \cC^m([0,\delta_i] \times M_i)
\; \longrightarrow \; \R \times \bigoplus_{i=j+1}^{j'-1} M_i\times M_i
$$
given by
$$
(z_j,\ldots,z_{j'},H_{j+1},\ldots,H_{j'-1})\longmapsto \bigl(
G_n(z_j,z_{j+1},z'_{j'-1},z_{j'}), ( \phi_{\delta_i}^{H_i}(z_i),z_i' )_{i=j+1,\ldots,j'-1} \bigr) .
$$
Hence the universal moduli space is a $\cC^m$ manifold if at every solution the operator
$$
(v_j,\ldots,v_{j'},h_{j+1},\ldots,h_{j'-1})\longmapsto \bigl(
dG_n(v_j,v_{j+1},v'_{j'-1},v_{j'}) , ( v_i' - D \phi_{\delta_i}^{H_i} (h_i,v_i)  )_{i=j+1,\ldots,j'-1} \bigr)
$$
is onto. Here surjectivity in the first component is guaranteed by the condition $d\tilde G_n(\ul{z})\neq 0$, and in the second component already $h_i\mapsto D \phi_{\delta_i}^{H_i} (h_i, 0)$ is surjective as in Theorem~\ref{H trans}. Now as before the implicit function theorem and Sard-Smale theorem, using $m> \dim M_j + \dim M_{j'} - 1$ to satisfy the index condition, provide a dense subset of $\bigoplus_{i=j+1}^{j'-1} \cC^m([0,\delta_i] \times M_i)$ for which
$\tilde G_n : \bigr(L_{j(j+1)}\times_{H_{j+1}} \ldots \times_{H_{j'-1}} L_{(j'-1)j'} \bigr) \cap \{ d\tilde G_n\neq 0 \} \to \R$ is transverse to $0$. 
Since this contains the lift of $G_{\ul{q},n}^{\ul{H},\ul{x}}: \cV^{\ul{H},\ul{x}}_{\ul{q},n} \to \R$, we find a dense open set of regular Hamiltonians of class $\cC^m$ for any given $\ul{q}\in S_{jj'}$, $1\leq n\leq N$, $\ul{x}\in\cI(\ul{L},\ul{H})$, and sufficiently large $m\in\N$.
Finally, $\cC^1$-small changes in $\ul{H}$ lead to small changes in intersection points and the linearized operators, hence we obtain open dense sets of regular values, and may take countable intersections to find a dense open subset $\cH_{jj'}(\ul{L})\subset\Ham^*(\ul{L})$ of regular smooth Hamiltonians.
\end{proof}

\section{Quilted Floer trajectories with constant components}

Given a cyclic generalized Lagrangian correspondence $\ul{L}$, widths $\ul{\delta}$, a regular tuple of Hamiltonian functions $\ul{H}\in\Ham^*(\ul{L})$, 
we now consider the Floer trajectories for some choice of almost complex structures $\ul{J}= (J_j)_{j=0,\ldots,r}\in\cJ_t(\ul{\delta})$.
For any pair $\ul{x}^-,\ul{x}^+\in\cI(\ul{L}, \ul{H})$ of generators and index $k\in\Z$, the moduli space of quilted Floer trajectories
$$
\cM^k(\ul{x}^-,\ul{x}^+;\ul{L},\ul{J}) :=
\bigl\{ \ul{u}=\bigl( u_j : \R \times [0,\delta_j] \to M_j \bigr)_{j=0,\ldots,r} \,\big|\, 
\eqref{Jjhol} , \eqref{ubc}, \eqref{ulim} , \Ind D_{\ul u} \dbar_{\ul{J}} =k \bigr\} / \R 
$$
is the space modulo simultaneous $\R$-shift of tuples of perturbed holomorphic strips
\begin{equation} \label{Jjhol}
\dbar_{J_j,H_j} u_j = \partial_s u_j + J_j \bigl(\partial_t u_j - X_{H_j}(u_j)  \bigr) = 0 
\qquad\forall j = 0,\ldots r ,
\end{equation}
satisfying the seam conditions
\begin{equation}\label{ubc}
(u_{j}({s},\delta_j),u_{j+1}({s},0)) \in L_{j(j+1)}\qquad\forall j = 0,\ldots r ,\ s\in\R
\end{equation}
as well as uniform limits 
\begin{equation} \label{ulim}
\lim_{s\to\pm\infty} u_j(s,\cdot) = x^\pm_j \quad\forall j = 0,\ldots,r  .
\end{equation}
Moreover, we fixed the index of the linearized operator -- as explained in the following.
By standard local action arguments, any such solution also has finite energy, and exponential decay analysis as in \cite{isom} shows that any solution is of Sobolev regularity $W^{1,p}$ relative to the limits for any $p>2$ in the following sense: 
We trivially extend $x_j^\pm$ to maps $\R\times[0,\delta_j]\to M_j$, then there exists $R>0$ such that $u_j(\pm s, t)$ takes values in an exponential ball around $x_j^\pm (s,t)$ for $\pm s > R$, and such that for each $j=0,\ldots,r$
\begin{equation} \label{W1p}
\bigl( (s,t) \mapsto \exp_{x_j^\pm(t)}^{-1}( u_j(\pm s, t) )  \bigr) 
\in W^{1,p}([R,\infty)\times[0,\delta_j], {x_j^\pm}^*TM_j) .
\end{equation}
With this, the moduli space can be identified with the $\R$-quotient of the zero set of a section $\dbar_{\ul{J}}:\cB\to\cE$ of a Banach bundle, where
$$
\cB :=
\bigl\{ \ul{u}=\bigl( u_j \in W^{1,p}_{\rm loc}(\R \times [0,\delta_j] , M_j ) \bigr)_{j=0,\ldots,r} \,\big|\, \eqref{ubc} , \eqref{W1p}  
 \bigr\} ,
$$
$\cE\to\cB$ is the Banach bundle with fibers 
$\cE_{\ul{u}} = \oplus_{j=0}^r L^p(\R\times[0,\delta_j], u_j^* TM_j)$,
and $\dbar_{\ul{J}}:\cB\to\cE$ is the ($\R$-invariant) Cauchy-Riemann operator
$\dbar_{\ul{J}}(\ul{u}) = \bigl( \overline\partial_{J_j,H_j} u_j \bigr)_{j=0,\ldots,r} $.
In \cite{quiltfloer} we proved that $\dbar_{\ul{J}}$ is a Fredholm section, and in the definition of the moduli space $\cM^k(\ul{x}^-,\ul{x}^+)$ we fix the Fredholm index of its linearization $D_{\ul{u}} \dbar_{\ul{J}}:T_{\ul{u}}\cB\to \cE_{\ul{u}}$.
In order to achieve transversality of the section $s$, we now restrict ourselves to a further dense open subset of Hamiltonian perturbations, as constructed in Section~\ref{Ham}.

\begin{definition}
Given a cyclic generalized Lagrangian correspondence $\ul{L}$ and widths $\ul{\delta}$, let
$$
\cH_{\reg}(\ul{L}) = \bigcap_{j,j'}  \cH_{jj'}(\ul{L}) \; \subset \; \cH_t(\ul{\delta})
$$
be the intersection over all pairs of indices $j\lhd  j'$ of the dense open subsets of regular Hamiltonians for some choices of covers of $\cS_{jj'}$ as in Theorem~\ref{intersection trans}.
\end{definition}

We now prove the main result.

\begin{theorem} \label{quilt trans} 
For any cyclic generalized Lagrangian correspondence $\ul{L}$ and any choice of widths $\ul{\delta}$ and regular Hamiltonians $\ul{H} \subset \cH_{\reg}(\ul{L})$, there exists a comeagre\footnote{
A subset of a topological space is {\it comeagre} if it is the intersection of countably many open dense subsets. Many authors in symplectic topology would use the term ``Baire second category'', which however in classical Baire theory \cite[Chapter 7.8]{royden} denotes more generally subsets that are not meagre, i.e.\ not the complement of a comeagre subset. Baire's Theorem applies to complete metric spaces such as the spaces of smooth almost complex structures considered here, and implies that every comeagre set is dense.
} 
subset $\cJ_{\reg}(\ul{L};\ul{H}) \subset \cJ_t(\ul{\delta})$ such that for all $\ul{J}\in\cJ_{\reg}(\ul{L};\ul{H})$, $\ul{x}^\pm\in\cI(\ul{L},\ul{H})$, and $k\in\Z$ the Cauchy-Riemann section $\dbar_{\ul{J}}:\cB\to\cE$ defined above is transverse to the zero section.
\end{theorem}

\begin{remark} \label{rmk trans}
In fact, we prove that for generic perturbation data $\ul{H} \subset \Ham^*(\ul{L})$ and $\ul{J}\in\cJ_t^{\reg}(\ul{L};\ul{H})$ any solution $\ul{u}\in\cM^k(\ul{x}^-,\ul{x}^+;\ul{L},\ul{J})$ with some constant components has split linearized seam conditions between constant and nonconstant components in the following sense: If $\partial_s u_j\not\equiv 0$ and $u_{j+1}(s,t)=x_{j+1}(t)$, then $T_{(u_j(s,\delta_j), x_{j+1}(0))} L_{j(j+1)} = \Lambda_j(s) \times \Lambda_{j+1}$
splits into two families of Lagrangian subspaces
\begin{align*}
\Lambda_j(s) &= {\rm Pr}_{TM_j}\bigl(  T_{(u_j(s,\delta_j), x_{j+1}(0))} L_{j(j+1)} \cap 
(T_{u_j(s,\delta_j)} M_{j}\times \{0\}) \bigr) , \\
\Lambda_{j+1} &= {\rm Pr}_{TM_{j+1}} \bigl( T_{(u_j(s,\delta_j), x_{j+1}(0))} L_{j(j+1)} \cap  (\{0\}\times T_{x_{j+1}(0)} M_{j+1}) \bigr),
\end{align*}
of which the second is constant.  The analogous statement holds for $\partial_s u_j\equiv 0$ and $\partial_s u_{j+1}\not\equiv 0$.
\end{remark}

\begin{proof}
Since $\cI(\ul{L},\ul{H})$ has finitely many elements (due to the transversality in Theorem~\ref{H trans} and compactness of the Lagrangian correspondences), and countable intersections of comeagre sets are comeagre in the complete metric space $\cJ_t(\ul{\delta})$, it suffices to consider a single pair $\ul{x}_\pm\in\cI(\ul{L},\ul{H})$ and indices $k\leq k_0$ for some fixed $k_0\in\N$.

The standard universal moduli space approach, using unique continuation for each strip separately, as discussed in the proof of \cite[Thm.5.2.4.]{quiltfloer}, provides a comeagre subset in $\cJ_t(\ul{L})$ for which the section $s$ is transverse at all zeros $\ul{u}$ for which $\partial_s u_i\not\equiv 0$ for all $i=0,\ldots,r$. In addition, $s$ is automatically transverse at any completely constant solution $\ul{u}\equiv \ul{x}^+=\ul{x}^-$. It remains to consider solutions $\ul{u}$ for which a proper subset $u_i$ for $i\in I\subset\{0,\ldots,r\}$ of components is constant\footnote{We call a component $u_i$ constant if $\partial_s u_i=0$, and hence $\partial_t u_i = X_{H_i}(u_i)$, so $u_i$ is a Hamiltonian trajectory in $t$, independent of $s$.}.
A necessary condition for such solutions to exist is $x^-_i=x^+_i$ for all $i\in I$, 
and hence the locally composed cyclic Lagrangian correspondences $\ul{L}^{I,\ul{H},\ul{x}^+}=\ul{L}^{I,\ul{H},\ul{x}^-}$ are the same. Now note that any solution $\ul{u}\in\cM^k(\ul{x}^-,\ul{x}^+;\ul{L},\ul{J})$ with the $I$ components constant induces a solution $(u_i)_{i\in I^C} \in\cM^k((x^-_i)_{i\in I^C},(x^+_i)_{i\in I^C};\ul{L}^{I,\ul{H},\ul{x}^\pm},(J_i)_{i\in I^C})$ in the moduli space of same index (see \cite[3.1.8]{quiltfloer} for the index calculation) for the locally composed correspondence.
Indeed, for consecutive pairs of indices $j\lhd j' \in I^C$ we have
$\bigl( u_j(s,\delta_j), x_{j+1}(0),x_{j+1}(\delta_{j+1}), \ldots,
x_{j'-1}(\delta_{j'-1}),u_{j'}(s,0)\bigr) \in \bigl(L_{j(j+1)}
\times_{H_{j+1}} \ldots \times_{H_{j'-1}} L_{(j'-1)j'}\bigr) \cap
\tilde\cU_{\ul{x},j,j'}$.  The converse is rarely true since the lifts
from $L_{jj'}^{\ul{H},\ul{x}^\pm}$ to $L_{j(j+1)} \times_{H_{j+1}}
\ldots \times_{H_{j'-1}} L_{(j'-1)j'}$ may not be constant.  Part of
this is encoded by the lift map $\cP_{jj'}:L_{jj'}^{\ul{H},\ul{x}^\pm} \to
M_{j+1} \times M_{j'-1}$ from Definition~\ref{def P}.  In the following
six steps we substantiate the intuition that Floer trajectories with
constant components for which automatic transversality fails are in
fact nongeneric solutions.  We denote by $\cJ_t^\ell(\ul{\delta})$ the
$\cC^\ell$-closure of $\cJ_t(\ul{\delta})$ in the topology of maps
$[0,\delta_j]\times TM_j \to TM_j$.

\medskip
\noindent
{\bf Step 0:} In preparation we need to fix some choices for each pair of indices $j\lhd j'$.

Firstly, we fix a metric on each $M_i$. Then, since $L_{jj'}^{\ul{H},\ul{x}^\pm}$ is compact, we may fix an open cover $\cP_{jj'}(L_{jj'}^{\ul{H},\ul{x}^\pm})\subset\bigcup_{\ul{p}\in C_{jj'}}\cW_{\ul{p}}$ 
by a finite number (indexed by $C_{jj'}\subset M_{j+1} \times M_{j'-1}$)
of exponential balls
$\cW_{\ul{p}}\subset M_{j+1} \times M_{j'-1}$ on which
$\exp_{\ul{p}}^{-1}:\cW_{\ul{p}}\to B_{\eps_{\ul{p}}}(0)\subset
T_{\ul{p}} (M_{j+1} \times M_{j'-1})$ is a
diffeomorphism.  We also fix a collection of $1$-dimensional subspaces
$(H_m)_{m=1,\ldots,\dim M_{j+1} + \dim M_{j'-1}}$ spanning $T_{\ul{p}}
(M_{j+1} \times M_{j'-1})$ for each $\ul{p}\in C_{jj'}$.

Secondly, as in Theorem~\ref{intersection trans} we fix a finite open cover $\cS_{jj'}\subset\bigcup_{\ul{q}\in S_{jj'}} \cV_{\ul{q}}$ and choose
$\ul{H}$ such that  the submanifold
$\bigl\{ \ul{z}\in \cV_{\ul{q}} \,\big|\, G_n(\ul{z})=0 ,  d G_n (\ul{z}) \neq 0 \bigr\}\subset L_{j(j+1)}\times L_{(j'-1)j'}$ is transverse to $\tilde L_{jj'}^{\ul{H},\ul{x}^\pm}$
for every $\ul{q}\in S_{jj'}$ and  $n=1,\ldots, N$.

\bigskip
\noindent
{\bf Step 1a :} 
We start by reviewing the regularity of the linearized operator at solutions without constant components, more precisely we prove the following:

{\it For every integer $\ell>k_0$ there exists a comeagre subset $\cJ_1^{\ell} \subset \cJ_t^\ell(\ul{\delta})$ such that for any $\ul{J}\in \cJ_1^\ell$ the linearized operator $D_{\ul{u}}\dbar_{\ul{J}}$ is surjective for all $\ul{u}\in\bigcup_{k\leq k_0}\cM^k(\ul{x}^-,\ul{x}^+;\ul{L},\ul{J})$ with no constant components.}

\medskip

This is what the arguments of \cite{quiltfloer} actually prove. To be precise, we consider the universal moduli problem $\cJ_t^\ell(\ul{\delta})\times\cB_{\rm nc}\to \cE|_{\cB_{\rm nc}}$, 
$(\ul{J},\ul{u})\mapsto \dbar_{\ul{J}}\ul{u}$ on the open subset 
$$
\cB_{\rm nc}:= \bigl\{ \ul{u}\in\cB \,\big|\, \partial_s u_i \neq 0 \;\forall i=0,\ldots,r \bigr\} \subset\cB.
$$
This is a $\cC^\ell$ section of a Banach bundle whose linearized operator at a zero $\dbar_{\ul{J}}\ul{u}=0$ is
$\bigl( \ul{K}=(K_i)_{i=0,\ldots,r}, \ul{\xi} \bigr) \mapsto (D_{\ul{u}}\dbar_{\ul{J}}) \ul{\xi} - \bigl(K_i J_i  \partial_s u_i \bigr)_{i=0,\ldots,r}$.
Here the second summand is already surjective by the same arguments as in \cite{fhs}. Indeed, 
the unique continuation theorem applies to the interior of every single nonconstant strip $u_i:\R\times(0,\delta_i)\to M_i$ and implies that the set of regular points, $(s_0,t_0)\in\R\times(0,\delta_i)$ with
$\partial_s u_i(s_0,t_0)\neq 0$ and $u_i^{-1}(u_i(\R\cup\{\pm\infty\}) , t_0
) =\{(s_0,t_0)\}$, is open and dense. 
So by the implicit function theorem $\{(\ul{J},\ul{u})\,|\, \dbar_{\ul{J}}\ul{u}=0\}$ is a $\cC^\ell$ Banach manifold. Its projection to $\cJ_t^\ell(\ul{\delta})$ is a Fredholm map of class $\cC^\ell$ and index $\Ind D_{\ul{u}} \leq k_0 \leq \ell -1$.
Hence, by the Sard-Smale theorem, the set of regular values, which coincides with the set $\ul{J}\in\cJ_t^\ell(\ul{\delta})$ such that $D_{\ul{u}}\dbar_{\ul{J}}$ is surjective for all solutions $\ul{u}$, is comeagre as claimed.

\bigskip
\noindent
{\bf Step 1b :} 
Next, a similar argument provides the following regularity of the linearized operator for the locally composed cyclic generalized Lagrangian correspondences:

{\it For every proper subset $I\subset\{0,\ldots,r\}$ such that
$(x^-_i)_{i\in I}=(x^+_i)_{i\in I}$ and every integer $\ell>k_0$ there exists a comeagre subset $\cJ_{1,I}^\ell \subset \cJ_t^\ell(\ul{\delta})$ such that for any $\ul{J}\in \cJ_{1,I}^\ell$ the linearized operator $D_{\ul{v}}\dbar_{\ul{J}^I}$ is surjective for all $\ul{v}\in\bigcup_{k\leq k_0}\cM^k((x^-_j)_{j\in I^C},(x^+_j)_{j\in I^C};\ul{L}^{I,\ul{H},\ul{x}^\pm},\ul{J}^I)$ with no constant components.}

\medskip

The locally composed cyclic Lagrangian correspondence $\ul{L}^{I,\ul{H},\ul{x}^\pm}$ consists of smooth, yet not compact, Lagrangian submanifolds. However, the compactness is not relevant for the universal moduli space arguments. Hence as in Step 1a we find a comeagre subset of the $\cC^\ell$-closure of $\oplus_{j\in I^C} \cC^\infty([0,\delta_j],\cJ(M_j,\omega_j))$ with the transversality properties. Then we let $\cJ_{1,I}^\ell$ be the preimage under the projection $\ul{J}\to \ul{J}^I=(J_j)_{j\in I^C}$.

\bigskip
\noindent
{\bf Step 2 :} 
In this first nonstandard step we show that for quilted Floer trajectories w.r.t.\ generic almost complex structures the differential $D\cP_{jj'}$ of the lift map from Definition~\ref{def P} vanishes along the seams bounding constant components. More precisely:

{\it For every integer $\ell>k_0$ and pair of indices $j\lhd j'$ 
such that $x^-_i=x^+_i$ for $i=j+1,\ldots,j'-1$
there exists a comeagre subset $\cJ_{2,j,j'}^\ell \subset \cJ_t^\ell(\ul{\delta})$ such that for any $\ul{J}\in \cJ_{2,j,j'}^\ell$ and $\ul{u}\in\bigcup_{k\leq k_0}\cM^k(\ul{x}^-,\ul{x}^+;\ul{L},\ul{J})$ with $u_{j+1},\ldots,u_{j'-1}$ constant we have 
${D_{(u_j(s,\delta_j), u_{j'}(s,0) )}\cP_{jj'} = 0}$ for all $s\in\R$.}

\medskip

Given $\ell,j,j'$ we set $I:=\{j+1,\ldots j'-1\}$ and start by proving
an intermediate Lemma, which asserts emptyness of the moduli spaces of quilted Floer trajectories for $\ul{L}^{I,\ul{H},\ul{x}^\pm}$ with $D\cP_{jj'}\not\equiv 0$ but a weak form of constant lifts $\cP_{jj'}$ at sufficiently many points along the seam.

\medskip
\noindent
{\bf Lemma for Step 2 :}  
{\it For every $k\leq k_0$, $\ul{p}\in C_{jj'}$, subspace 
$H_m\subset T_{\ul{p}}(M_{j+1} \times M_{j'-1})$, and tuple of rationals 
$s_0<\ldots<s_{k+1}\in\Q$ there exists a comeagre subset 
$\cJ(\ul{p},m,s_0,\ldots,s_{k+1}) \subset \cJ_t^\ell(\ul{\delta})$ such that for any
$\ul{J}\in \cJ(\ul{p},m,s_0,\ldots,s_{k+1})$ there exists no Floer trajectory 
$\ul{v}\in\cM^k((x^-_i)_{i\in I^C},(x^+_i)_{i\in I^C};\ul{L}^{I,\ul{H},\ul{x}^\pm},\ul{J}^I)$ which satisfies  
$\cP_{jj'}  (v_j(s_l,\delta_j), v_{j'}(s_l,0) ) \in \cW_{\ul{p}}$ and 
${\rm Pr}_{H_m}\circ D\exp_{\ul{p}}^{-1} \circ \, D_{(v_j(s_l,\delta_j), v_{j'}(s_l,0) )}\cP_{jj'}  \neq 0 $ 
for all $l=0,\ldots,k+1$, and moreover for $l=1,\ldots,{k+1}$
$$
{\rm Pr}_{H_m} \exp_{\ul{p}}^{-1}(\cP_{jj'}(v_j(s_l,\delta_j), v_{j'}(s_l,0))) =
{\rm Pr}_{H_m} \exp_{\ul{p}}^{-1}(\cP_{jj'}(v_j(s_0,\delta_j), v_{j'}(s_0,0))) .
$$
}

\medskip

The proof is by a universal moduli space argument.  
Let $\cB$ be the Banach manifold as in the definition of
$\cM^k((x^-_i)_{i\in I^C},(x^+_i)_{i\in I^C};\ul{L}^{I,\ul{H},\ul{x}^\pm},\ul{J}^I)$. Then
$$
\cB':=\left\{ \ul{v}\in\cB \,\left|\, 
\begin{array}{l}
\cP_{jj'}  (v_j(s_l,\delta_j), v_{j'}(s_l,0) ) \in \cW_{\ul{p}} \quad \forall\; 0\leq l\leq k+1 ,  \\
{\rm Pr}_{H_m}\circ D\exp_{\ul{p}}^{-1} \circ \, D_{(v_j(s_l,\delta_j), v_{j'}(s_l,0) )}\cP_{jj'}  \neq 0 \quad \forall\; 0\leq l\leq k+1 
 \end{array}
\right\} \right.
$$
is an open subset of $\cB$, and 
$$
s(\ul{J},\ul{v}):=\left( \; \dbar_{\ul{J}^I}(\ul{v}) \;,\; 
\left(  \begin{array}{l}     
{\rm Pr}_{H_m} \bigl( \exp_{\ul{p}}^{-1}(\cP_{jj'}(v_j(s_l,\delta_j), v_{j'}(s_l,0))) \bigr) \\
- {\rm Pr}_{H_m} \bigl( \exp_{\ul{p}}^{-1}(\cP_{jj'}(v_j(s_0,\delta_j), v_{j'}(s_0,0))) \bigr)
\end{array}
\right)_{l=1,\ldots,k+1} \right)
$$
defines a $\cC^\ell$ section of the bundle $\cE|_{\cB'}\times ( H_m )^{k+1} \to \cJ_t^\ell(\ul{\delta})\times\cB'$.
Its linearized operator at a zero maps $\bigl( \ul{K}=(K_i)_{i=0,\ldots,r}, \ul{\xi}=(\xi_i)_{i\in I^C} \bigr)$ to
$$
\left( \begin{array}{c} 
\bigl(D_{\ul{v}}\dbar_{\ul{J}^I}\bigr) \ul{\xi} - \bigl(K_i J_i  \partial_s v_i \bigr)_{i\in I^C} 
\phantom{\bigg|} \\
\left(  \begin{array}{l}   
\bigl( {\rm Pr}_{H_m} \circ D\exp_{\ul{p}}^{-1} \circ  D_{(v_j(s_l,\delta_j), v_{j'}(s_l,0)) } \cP_{jj'}  \bigr)
\bigl(\xi_j(s_l,\delta_j), \xi_{j'}(s_l,0) \bigr)  \\
- \bigl( {\rm Pr}_{H_m} \circ D\exp_{\ul{p}}^{-1} \circ  D_{(v_j(s_0,\delta_j), v_{j'}(s_0,0)) } \cP_{jj'}  \bigr)
\bigl(\xi_j(s_0,\delta_j), \xi_{j'}(s_0,0) \bigr) 
\end{array}
\right)_{l=1,\ldots,k+1}
\end{array}
\right).
$$
Here the second summand in the first component is surjective by the same arguments as in Step 1a, using just the freedom in $\ul{K}$. The second component is surjective since by definition of $\cB'$ each map
$$ 
{\rm Pr}_{H_m} \circ D \exp_{\ul{p}}^{-1} \circ D_{(v_j(s_l,\delta_j), v_{j'}(s_l,0)) } \cP_{jj'} : T_{(v_j(s_l,\delta_j), v_{j'}(s_l,0)) } L^{\ul{H},\ul{x}^\pm}_{jj'} \to H_m
$$ 
is nonzero, i.e.\ surjective onto this one dimensional subspace, and
$\ul{\xi}$ can be chosen to assume any given tuple of values on the
linearized Lagrangian correspondence $TL^{\ul{H},\ul{x}^\pm}_{jj'}$ at
distinct $s_1,\ldots, s_{k+1} \in \R$.  So by the implicit function
theorem $\{(\ul{J},\ul{v})\,|\, s(\ul{J},\ul{v})=0\}$ is a $\cC^\ell$
Banach manifold. Its projection to $\cJ_t^\ell(\ul{\delta})$ is a
Fredholm map of class $\cC^\ell$ and index $\Ind D_{\ul{v}} - (k+1) =
-1$.  Hence, by the Sard-Smale theorem, the set of regular values is
comeagre.  Finally, since the index is negative, the set of solutions
for a regular $\ul{J}$ is empty, which proves the Lemma.

\medskip

We now obtain a comeagre subset $\cJ_{2,j,j'}^\ell \subset \cJ_t^\ell(\ul{\delta})$ by taking the countable intersection of the comeagre sets $\cJ(\ul{p},m,s_0,\ldots,s_{k+1})$ given by the Lemma.
Then suppose by contradiction that for some
$\ul{J}\in \cJ_{2,j,j'}^\ell$ we have a solution
$\ul{u}\in\cM^k(\ul{x}^-,\ul{x}^+;\ul{L},\ul{J})$ for some $k\leq k_0$
with $u_{j+1},\ldots,u_{j'-1}$ constant but $D_{(u_j(s_0,\delta_j),
  u_{j'}(s_0,0) )}\cP_{jj'} \neq 0$ for some $s_0\in\R$.  
As discussed at the beginning of the proof of this Theorem, $\ul{u}$ induces a solution $\ul{v}:= (u_i)_{i\in I^C}\in \cM^k((x^-_i)_{i\in I^C},(x^+_i)_{i\in I^C};\ul{L}^{I,\ul{H},\ul{x}^\pm},\ul{J}^I)$ such that $\cP_{jj'}(v_j(s,\delta_j),v_{j'}(s,0))=(u_{j+1}(s,0),u_{j'-1}(s,\delta_{j'-1}))$ is constant in $s\in\R$.  Moreover, we have $D_{(v_j(s_0,\delta_j), v_{j'}(s_0,0))}\cP_{jj'} \neq 0$ for some $s_0\in\R$.  Since this is an open condition, we may also find $s_0\in\Q$ with the same nonvanishing.
Then we have 
$\cP_{jj'}(v_j(s_0,\delta_j), v_{j'}(s_0,0) )\in \cW_{\ul{p}}$ for some $\ul{p}\in C_{jj'}$ and ${\rm Pr}_{H_m}\circ D\exp_{\ul{p}}^{-1} \circ \, D_{(v_j(s_0,\delta_j), v_{j'}(s_0,0) )}\cP_{jj'} \neq 0$ for one
of the spanning subspaces $H_m\subset T_{\ul{p}}(M_{j+1} \times M_{j'-1})$.  Again, these are open conditions, so we may find rational numbers $s_0<s_1<\ldots<s_{k+1}$ with the same properties. This contradicts the Lemma since 
$\cP_{jj'}(v_j(s_l,\delta_j),v_{j'}(s_l,0))
=(x_{j+1}(0),x_{j'-1}(\delta_{j'-1}))
=\cP_{jj'}(v_j(s_0,\delta_j), v_{j'}(s_0,0))$ 
for $l=1,\ldots,{k+1}$.

\bigskip
\noindent
{\bf Step 3 :} 
Extending Step 2, we show that for quilted Floer trajectories w.r.t.\ generic almost complex structures in fact the splitting condition of Proposition~\ref{prop L} on the linearized seam conditions holds along the seams bounding constant components. More precisely:

{\it For every integer $\ell>k_0$ and pair of indices
  $j\lhd j'$ such that ${x^-_i=x^+_i=:x_i}$ for $i=j+1,\ldots,j'-1$
  there exists a comeagre subset $\cJ_{3,j,j'}^\ell \subset
  \cJ_t^\ell(\ul{\delta})$ such that for any $\ul{J}\in
  \cJ_{3,j,j'}^\ell$ and $\ul{u}\in\bigcup_{k\leq
    k_0}\cM^k(\ul{x}^-,\ul{x}^+;\ul{L},\ul{J})$ with
  $u_{j+1},\ldots,u_{j'-1}$ constant the intersection
$$
T_{(u_j(s,\delta_j), x_{j+1}(0),x_{j'-1}(\delta_{j'-1}),u_{j'}(s,0))}
(L_{j(j+1)}\times L_{(j'-1)j'}) \cap ( \{0\}\times T_{x_{j+1}(0)} M_{j+1}\times T_{x_{j'-1}(\delta_{j'-1})}M_{j'-1}\times\{0\} \bigr)
$$
projects to a Lagrangian subspace $T_{x_{j+1}(0)} M_{j+1}\times T_{x_{j'-1}(\delta_{j'-1})} M_{j'-1}$ that is independent of $s\in\R$.}

\medskip

Given $\ell,j,j'$ we set $I:=\{0,\ldots,j,j',\ldots,r\}$ and start by
proving an intermediate Lemma which asserts emptyness of the moduli spaces of quilted Floer trajectories for $\ul{L}^{I,\ul{H},\ul{x}^\pm}$ with $dG_{n}\not\equiv 0$ but $G_n=0$ at sufficiently many points along the seam.
This will be relevant since by Theorem~\ref{intersection trans} the split locus is locally given by the intersection of the zero sets $G_n^{-1}(0)$, and since $d G_n \equiv 0$ along a path in the split locus ensures $s$-independence of the Lagrangian subspace of $TM_{j+1}\times TM_{j'-1}$ that arises from the splitting.

\medskip
\noindent
{\bf Lemma for Step 3 :}  {\it For every $k\leq k_0$,  $\ul{q}\in S_{jj'}$, $1\leq n\leq N$, and tuple of rationals $s_0<\ldots<s_k\in\Q$ there exists a comeagre subset $\cJ(\ul{q},n,s_0,\ldots,s_{k}) \subset \cJ_t^\ell(\ul{\delta})$ such that for $\ul{J}\in \cJ(\ul{q},n,s_0,\ldots,s_{k})$ there exists no solution $\ul{v}\in\cM^k((x^-_i)_{i\in I^C},(x^+_i)_{i\in I^C};\ul{L}^{I,\ul{H},\ul{x}^\pm},\ul{J}^I)$ with ${v}_{jj'}(s_l):=(v_j(s_l,\delta_j),\cP_{jj'}(v_j(s_l,\delta_j), v_{j'}(s_l,0) ),v_{j'}(s_l,0))\in \cV_{\ul{q}}$,  $G_n({v}_{jj'}(s_l))=0$, and $d G_n({v}_{jj'}(s_l)) \neq 0$ for $0\leq l \leq k$.
}

\medskip

Let $\cB$ be the Banach manifold as in the definition of 
$\cM^k((x^-_i)_{i\in I^C},(x^+_i)_{i \in I^C};\ul{L}^{I,\ul{H},\ul{x}^\pm},\ul{J}^I)$ and recall from Theorem~\ref{intersection trans}~(b) the transversality of the function
$G_{\ul{q},n}^{\ul{H},\ul{x}^\pm}: \cV^{\ul{H},\ul{x}^\pm}_{\ul{q},n} \to \R$, $(z_j,z_{j'})\mapsto G_n (z_j,\cP_{jj'}(z_j,z_{j'}),z_{j'})$ on the open set
$$
\cV^{\ul{H},\ul{x}^\pm}_{\ul{q},n}:=\bigl\{ (z_j,z_{j'})\in L_{jj'}^{\ul{H},\ul{x}^\pm} \,\big|\, (z_j,\cP_{jj'}(z_j,z_{j'}),z_{j'})\in \cV_{\ul{q}}, \; d G_n (z_j,\cP_{jj'}(z_j,z_{j'}),z_{j'}) \neq 0 \bigr\} .
$$
Then 
$\cB'':=\bigl\{ \ul{v}\in\cB \,\big|\, 
\bigl( v_j(s_l,\delta_j), v_{j'}(s_l,0) \bigr) \in \cV^{\ul{H},\ul{x}^\pm}_{\ul{q},n} \;\; \forall\; 0\leq l\leq k  \bigr\}$
is an open subset of $\cB$ and 
$$
s(\ul{J},\ul{v}):=\Bigl( \; \dbar_{\ul{J}^I}(\ul{v}) \;,\;  
\bigl( G_{\ul{q},n}^{\ul{H},\ul{x}^\pm} \bigl( v_j(s_l,\delta_j), v_{j'}(s_l,0) \bigr)_{l=0,\ldots,k} \Bigr)
$$
defines a $\cC^\ell$ section of the bundle $\cE|_{\cB''}\times \R^{k+1} \to \cJ_t^\ell(\ul{\delta})\times\cB''$.
Its linearized operator at a zero maps $\bigl( \ul{K}=(K_i)_{i=0,\ldots,r}, \ul{\xi}=(\xi_i)_{i\in I^C} \bigr)$ to
$$
\Bigl( \; \bigl(D_{\ul{v}}\dbar_{\ul{J}^I}\bigr) \ul{\xi} - \bigl(K_i J_i  \partial_s v_i \bigr)_{i\in I^C} \;,\;
\bigl( D_{( v_j(s_l,\delta_j),v_{j'}(s_l,0))}  G_{\ul{q},n}^{\ul{H},\ul{x}^\pm}  \bigl( \xi_j(s_l,\delta_j),\xi_{j'}(s_l,0) \bigr)_{l=0,\ldots,k} \Bigl).
$$
As before, the second summand in the first component is surjective using just the freedom in $\ul{K}$. The second component is surjective since by Theorem~\ref{intersection trans}~(b) each map $D_{( v_j(s_l,\delta_j),v_{j'}(s_l,0))}  G_{\ul{q},n}^{\ul{H},\ul{x}^\pm}$ is surjective, and $\ul{\xi}$ can be chosen to assume any given tuple of values on the linearized Lagrangian correspondence $TL^{\ul{H},\ul{x}^\pm}_{jj'}$ at distinct $s_0,\ldots, s_{k+1} \in \R$.
So by the implicit function theorem $\{(\ul{J},\ul{v})\,|\, s(\ul{J},\ul{v})=0\}$ is a $\cC^\ell$ Banach manifold and its projection to $\cJ_t^\ell(\ul{\delta})$ is a Fredholm map of class $\cC^\ell$ and negative index $\Ind D_{\ul{v}} - (k+1) = -1$. As before, by the Sard-Smale theorem, the set of regular values is comeagre, and for each regular $\ul{J}$ the set of solutions is empty. This proves the Lemma.

\medskip

We now obtain a comeagre subset $\cJ_{3,j,j'}^\ell \subset \cJ_t^\ell(\ul{\delta})$ by taking the countable intersection of the comeagre sets $\cJ(\ul{q},n,s_0,\ldots,s_{k})$ given by the Lemma with $\cJ_{2,j,j'}^\ell$.
Now consider any $\ul{J}\in \cJ_{3,j,j'}^\ell$ and $\ul{u}\in\bigcup_{k\leq k_0}\cM^k(\ul{x}^-,\ul{x}^+;\ul{L},\ul{J})$ with $u_{j+1},\ldots,u_{j'-1}$ constant.
As before, this induces a solution $\ul{v}:= (u_i)_{i\in I^C}\in \cM^k((x^-_i)_{i\in I^C},(x^+_i)_{i\in I^C};\ul{L}^{I,\ul{H},\ul{x}^\pm},\ul{J}^I)$ such that $\cP_{jj'}(v_j(s,\delta_j), v_{j'}(s,0))=(u_{j+1}(s,0),u_{j'-1}(s,\delta_{j'-1}))=(x_{j+1}(0),x_{j'-1}(\delta_{j'-1}))$ is independent of $s\in\R$.
Moreover, we know from Step 2 that $D_{(v_j(s_0,\delta_j), v_{j'}(s_0,0) )}\cP_{jj'}= 0$ for all $s\in\R$, and hence by Proposition~\ref{prop L} the intersection
$\Lambda(v_{jj'}(s))$ at $v_{jj'}(s):=(v_j(s,\delta_j), x_{j+1}(0),x_{j'-1}(\delta_{j'-1}),v_{j'}(s,0) )$
induces a Lagrangian subspace of $T_{x_{j+1}(0)} M_{j+1}\times T_{x_{j'-1}(\delta_{j'-1})} M_{j'-1}$
for every $s\in\R$.
Suppose by contradiction that it is not constant on any neighbourhood of $\sigma\in\R$.
Fix $\ul{q}\in S_{jj'}$ such that $v_{jj'}(s)\in\cV_{\ul{q}}$ for $|s-\sigma|<\eps$ sufficiently small, then by Theorem~\ref{intersection trans}~(a) we have $G_1(v_{jj'}(s)) =\ldots = G_N(v_{jj'}(s)) =0$ for all $|s-\sigma|<\eps$, but $dG_{n}(v_{jj'}(\sigma'))\neq 0$ for some $1\leq n\leq N$ and $\sigma'\in (\sigma-\eps,\sigma+\eps)$.
Since the nonvanishing is an open condition, we may also find $s_0<\ldots<s_k\in\Q\cap (\sigma-\eps,\sigma+\eps)$ with $dG_{n}(v_{jj'}(s_l))\neq 0$, in contradiction to 
the Lemma.

\bigskip
\noindent
{\bf Step 4 :} 
Next we explicitly state Step 3 as a splitting property and deduce surjectivity of part of the linearized operator:

{\it If $\ul{u}\in\cM^k(\ul{x}^-,\ul{x}^+;\ul{L},\ul{J})$ with $u_i(s,t)=x_i^\pm(t)=:x_i(t)$ for $i=j+1,\ldots,j'-1$ gives rise to a constant Lagrangian subspace as in Step 3, then the linearized seam conditions
$$
T_{(u_j(s,\delta_j), u_{j+1}(s,0))} L_{j(j+1)} = \Lambda_j(s) \times \Lambda_{j+1} , \qquad
T_{(u_{j'-1}(s,\delta_{j'-1}),u_{j'}(s,0))} L_{(j'-1)j'} = \Lambda_{j'-1} \times \Lambda_{j'}(s)
$$
split into the Lagrangian subspaces
\begin{align*}
\Lambda_j(s) \; &\cong \;\qquad T_{(u_j(s,\delta_j), x_{j+1}(0))} L_{j(j+1)} \cap \bigl(T_{u_j(s,\delta_j)} M_{j}\times \{0\}\bigr) \qquad\quad\, \hookrightarrow \; T_{u_j(s,\delta_j)} M_{j} \\
\Lambda_{j+1} \;&\cong \;\qquad T_{(u_j(s,\delta_j), x_{j+1}(0))} L_{j(j+1)} \cap  \bigl( \{0\}\times T_{x_{j+1}(0)} M_{j+1} \bigr)  \qquad\; \hookrightarrow \; T_{x_{j+1}(0)} M_{j+1}\\
\Lambda_{j'-1}\;&\cong\; T_{(x_{j'-1}(\delta_{j'-1}), u_{j'}(s,0))} L_{(j'-1)j'} \cap \bigl( T_{x_{j'-1}(\delta_{j'-1})} M_{j'-1}\times \{0\} \bigr)  \;\hookrightarrow \; T_{x_{j'}(\delta_{j'})} M_{j'}\\
\Lambda_{j'}(s) \;&\cong \;T_{(x_{j'-1}(\delta_{j'-1}), u_{j'}(s,0))} L_{(j'-1)j'} \cap \bigl( \{0\}\times T_{u_{j'}(s,0)} M_{j'} \bigr) \qquad\;\;\; \hookrightarrow \; T_{u_{j'}(s,0)} M_{j'} .
\end{align*}
Moreover,  the operator
$D_{jj'}:=\bigl( \partial_s + J_i(x_i) \partial_t  - J_i(x_i) D_{x_i}X_{H_i}  \bigr)_{i=j+1,\ldots, j'-1}$
maps
$$
\left\{
\ul{\xi}\in \oplus_{i=j+1}^{j'-1} W^{1,p}(\R\times[0,\delta_i], x_i^*TM_i) \left| 
\begin{array}{l}
\bigl( \xi_i(s,\delta_i) , \xi_{i+1}(s,0) \bigr) \in T_{(x_i(\delta_i),x_{i+1}(0))} L_{i(i+1)} \quad\forall i
\\
\quad \xi_{j+1}(s,0) \in \Lambda_{j+1}, \qquad
\xi_{j'-1}(s,\delta_{j'-1}) \in \Lambda_{j'-1}
\end{array}
\right\}\right. 
$$
onto $\oplus_{i=j+1}^{j'-1} L^p(\R\times[0,\delta_i], x_i^*TM_i)$.
}

\medskip

We can express the operator $D_{jj'} = \partial_s + A$ in terms of an $s$-independent operator $A= ( J_i(x_i) \partial_t - J_i(x_i) D_{x_i}X_{H_i} )_{i=j+1,\ldots,j'-1}$, which is self-adjoint on $\oplus_{i=j+1}^{j'-1} L^2([0,\delta_i], x_i^*TM_i)$ with domain 
$$
\left\{
\ul{\zeta}\in \oplus_{i=j+1}^{j'-1} W^{1,2}([0,\delta_i], x_i^*TM_i) \left| 
\begin{array}{l}
\bigl( \zeta_i(\delta_i) , \zeta_{i+1}(0) \bigr) \in T_{(x_i(\delta_i),x_{i+1}(0))} L_{i(i+1)} \quad\forall i
\\
\quad \zeta_{j+1}(0) \in \Lambda_{j+1}, \qquad
\zeta_{j'-1}(\delta_{j'-1}) \in \Lambda_{j'-1}
\end{array}
\right\}\right. .
$$
Moreover, the nondegeneracy of the intersection points $\cI(\ul{L},\ul{H})$ implies that $A$ is invertible. Indeed, the linearized operator cutting out $\cI(\ul{L},\ul{H})$ as trajectory space splits at $\ul{x}^\pm$ into $A$ and $( J_i(x^\pm_i) \partial_t - J_i(x^\pm_i) D_{x^\pm_i}X_{H_i} )_{i\in\{j+1,\ldots,j'-1\}^C}$ with the analogous linearized seam conditions.
Now a general spectral analysis and Sobolev embedding argument proves that $D_{jj'}$ is in fact an isomorphism, see e.g.\ \cite[Chapter 3]{donaldson:floer}.

\bigskip
\noindent
{\bf Step 5 :}
We deduce from the previous steps that the set of almost complex structures of class $\cC^\ell$, for which the linearized operators are surjective, is dense in the following sense:

{\it For every integer $\ell>k_0$ let $\cJ_{\rm reg}^\ell$ be the set of $\ul{J}\in \cJ_t^\ell(\ul{\delta})$ for which the linearized operators $D_{\ul{u}}\dbar_{\ul{J}}$ are surjective at all $\ul{u}\in\bigcup_{k\leq k_0}\cM^k(\ul{x}^-,\ul{x}^+;\ul{L},\ul{J})$. Then $\cJ_{\rm reg}^{\ell}\subset \cJ_t^\ell(\ul{\delta})$ is dense. }

\medskip

The density will follow from proving that $\cJ_{\rm reg}^{\ell}$ contains the intersection of $\cJ^\ell_1$, all $\cJ^\ell_{1,I}$, and all $\cJ^\ell_{3,j,j'}$, i.e.\ a comeagre and hence dense set. So we need to consider a given $\ul{J}\in \cJ^\ell_1\cap \bigcap_{I} \cJ^\ell_{1,I}\cap \bigcap_{j,j'}\cJ^\ell_{3,j,j'}$ and show surjectivity of $D_{\ul{u}}\dbar_{\ul{J}}$ for all solutions $\ul{u}$.

Step~1a ensures surjectivity if $\ul{u}$ has no constant components, so
it remains to consider $\ul{u}\in\bigcup_{k\leq
  k_0}\cM^k(\ul{x}^-,\ul{x}^+;\ul{L},\ul{J})$ with $\partial_s u_i
\equiv 0 \Leftrightarrow i\in I$ for some subset
$I\subset\{0,\ldots,r\}$.  If all components are constant, then
surjectivity follows as in Step 4 from the fact that
$D_{\ul{u}}\dbar_{\ul{J}} = \partial_s + A$ is given by the
$s$-independent self-adjoint operator $A= ( J_i(x_i) \partial_t -
J_i(x_i) D_{x_i}X_{H_i} )_{i=0,\ldots,r}$ with constant Lagrangian
seam conditions.

If $\ul{u}$ is a solution with constant components for a proper subset $I\subset\{0,\ldots,r\}$, then as before this induces a solution $(u_j)_{j\in I^C}\in \cM^k((x^-_j)_{j\in I^C},(x^+_j)_{j\in I^C};\ul{L}^{I,\ul{H},\ul{x}^\pm},\ul{J}^I)$ for the locally composed cyclic Lagrangian correspondence $\ul{L}^{I,\ul{H},\ul{x}^\pm}$ consisting of $L_{j(j+1)}$ for $j,j+1\in I^C$ and $L_{jj'}^{\ul{H},\ul{x}^\pm}\subset M_j\times M_{j'}$ for each pair of consecutive indices $j\lhd j'\in I^C$, i.e.\ with $j+1,\ldots,j'-1\in I$. Moreover, Step 3 implies that the linearized seam conditions at each consecutive $j\lhd j'\in I^C$ split as in Step 4. As a direct consequence, the locally composed correspondence also splits,
$$
T_{(u_j(s,\delta_j),u_{j'}(s,0)) } L_{jj'}^{\ul{H},\ul{x}^\pm} = \Lambda_j(s) \times \Lambda_{j'}(s) .
$$
That is, the seam conditions in the linearized operator $D_{(u_j)_{j\in I^C}}\dbar_{\ul{J}^I}$ for the moduli space associated to the local composition coincide with the seam conditions in nonconstant components of the linearized operator $D_{\ul{u}}\dbar_{\ul{J}}$. Hence the linearized operator for the full problem $D_{\ul{u}}\dbar_{\ul{J}}$ is the direct sum of $D_{(u_j)_{j\in I^C}}\dbar_{\ul{J}^I}$ and the operators $D_{jj'}$ as in Step 4 for each consecutive $j\lhd j'\in I^C$. The latter are surjective by Step 4, whereas the former is surjective by Step 1b.
This shows that $D_{\ul{u}}\dbar_{\ul{J}}$ is indeed surjective for all solutions $\ul{u}$ of index up to $k_0$, and hence $\ul{J}\in\cJ_{\rm reg}^{\ell}$.

\bigskip
\noindent
{\bf Step 6:} 
As final step we use an intersection argument due to Taubes to transfer from 
$\cC^\ell$ 
to 
$\cC^\infty$ almost complex structures. For fixed $\ul{x}^\pm\in\cI(\ul{L},\ul{H})$, $k_0\in\N_0$ this proves the following:

 {\it  Let 
 $\cJ_{\reg, k_0}(\ul{x}^-,\ul{x}^+)$ 
 be the set of $\ul{J} \in \cJ_t(\ul{\delta})$ for which the linearized operators $D_{\ul{u}}\dbar_{\ul{J}}$ are surjective for all $\ul{u}\in\bigcup_{k\leq k_0}\cM^k(\ul{x}^-,\ul{x}^+;\ul{L},\ul{J})$.
Then $\cJ_{\reg, k_0}(\ul{x}^-,\ul{x}^+) \subset \cJ_t(\ul{\delta})$ is comeagre.}

\medskip

For every $R\geq 0$ let $\cJ_{\rm reg}^{R}\subset \cJ_t(\ul{\delta})$ and $\cJ_{\rm reg}^{\ell,R}\subset \cJ_t^\ell(\ul{\delta})$ for $\ell>k_0$ be the sets of $\ul{J}$ for which the linearized operators $D_{\ul{u}}\dbar_{\ul{J}}$ are surjective at all $\ul{u}\in\bigcup_{k\leq k_0}\cM^k(\ul{x}^-,\ul{x}^+;\ul{L},\ul{J})$ with $\|\partial_s\ul{u}\|_\infty:=\max_i\|\partial_s u_i\|_\infty \leq R$.

Then $\cJ_{\rm reg}^{R}$ and $\cJ_{\rm reg}^{\ell,R}$ are open in the $\cC^\infty$- resp.\ $\cC^\ell$-topology, by the following compactness and gluing argument as in \cite{fhs}: Suppose by contradiction that $\ul{J}^\nu\to \ul{J}^\infty \in\cJ_{\rm reg}^{\ell,R}$ in the $\cC^1$-topology but $D_{\ul{u}^\nu}\dbar_{\ul{J}^\nu}$ fails to be surjective for some solutions $\dbar_{\ul{J}^\nu}\ul{u}^\nu=0$ with $\|\partial_s\ul{u}^\nu\|_\infty\leq R$.
Then a subsequence of $\ul{u}^\nu$ converges to a broken trajectory, consisting of a finite number of nonconstant solutions with respect to $\ul{J}^\infty$, and satisfying the same uniform derivative bound. These components cannot have negative index since $\ul{J}^\infty$ is regular for indices up to $k_0$. So, by index additivity, all components of the broken trajectory have index at most $k_0$, and thus the linearized operators at these solutions are surjective. Now a standard gluing construction shows that in fact $D_{\ul{u}^\nu}\dbar_{\ul{J}^\nu}$ must be surjective for some large~$\nu$.

We moreover know that $\cJ_{\rm reg}^{\ell,R}\subset \cJ_t^\ell(\ul{\delta})$ is dense since it contains the dense set $\cJ_{\rm reg}^\ell$ from Step 5. Now $\cJ_{\rm reg}^{R}\subset \cJ_t(\ul{\delta})$ is dense in the $\cC^\infty$-topology since $\cJ_{\rm reg}^{R} = \cJ_{\rm reg}^{\ell,R}\cap\cJ_t(\ul{\delta})$, where $\cJ_{\rm reg}^{\ell,R}\subset \cJ^\ell_t(\ul{\delta})$ is open and dense in the $\cC^\ell$-topology for all $\ell > k_0$.
Finally $\cJ_{\reg, k_0}(\ul{x}^-,\ul{x}^+)=\bigcap_{R\in\N} \cJ_{\rm reg}^{R}$ is a countable intersection of open dense subsets, i.e.\  comeagre.

\medskip

To finish the proof of the Theorem, let $\cJ_{\rm reg}(\ul{L};\ul{H})$ be the set of $\ul{J} \in \cJ_t(\ul{\delta})$ for which the linearized operators $D_{\ul{u}}\dbar_{\ul{J}}$ are surjective for all $\ul{u}\in\cM^k(\ul{x}^-,\ul{x}^+;\ul{L},\ul{J})$ with $\ul{x}^\pm\in\cI(\ul{L},\ul{H})$ and $k\in\Z$.
Then $\cJ_{\rm reg}(\ul{L};\ul{H}) = \bigcap_{k_0\in\N_0}\bigcap_{\ul{x}^\pm\in\cI(\ul{L},\ul{H})}\cJ_{\reg, k_0}(\ul{x}^-,\ul{x}^+)$  is comeagre in $\cJ_t(\ul{\delta})$ since it is the countable intersection of comeagre sets.
\end{proof}

\end{document}